\documentclass[oneside]{amsart}

\usepackage{amsmath, amsthm, amsfonts}
\usepackage{amssymb}
\usepackage[utf8]{inputenc}
\usepackage{latexsym}
\usepackage{verbatim}
\usepackage{url}
\usepackage{textcomp}
\usepackage[all,cmtip]{xy}
\usepackage{color}
\usepackage{hyperref}
\usepackage{tabularx}
\input{xypic}
\usepackage{enumerate}
\usepackage{amssymb}
\usepackage{array}

\DeclareMathAlphabet{\mathfr}{U}{euf}{m}{n}

\newtheorem{theorem}{Theorem}[section]
\newtheorem{conjecture}[theorem]{Conjecture}
\newtheorem{proposition}[theorem]{Proposition}
\newtheorem{corollary}[theorem]{Corollary}
\newtheorem{hypothesis}[theorem]{Hypothesis}
\newtheorem{lemma}[theorem]{Lemma}

\newtheorem{definition}[theorem]{Definition}

\theoremstyle{remark}
\newtheorem{remark}[theorem]{Remark}

\newcommand\acc[2]{\ensuremath{{}^{#1}\hskip-0.3ex{#2}}}

\newcommand{\Q}{\mathbb Q}
\newcommand{\Qbar}{{\overline{\mathbb Q}}}

\newcommand{\Gal}{\mathrm{Gal}}

\newcommand{\C}{\mathbb C}

\newcommand{\GL}{\mathrm{GL}}
\newcommand{\M}{\mathrm{M }}

\newcommand{\End}{\operatorname{End}}
\newcommand{\Hom}{\operatorname{Hom}}
\newcommand{\Frob}{\operatorname{Frob}}
\newcommand{\USp}{\operatorname{USp}}
\newcommand{\Ind}{\operatorname{Ind}}

\newcommand{\Aut}{\operatorname{Aut}}
\newcommand{\Zar}{\operatorname{Zar}}

\newcommand{\p}{\mathfrak{p}}
\newcommand{\Rr}{\mathcal{R}}

\newcommand{\GSp}{\operatorname{GSp}}
\newcommand{\Tr}{\operatorname{Tr}}
\newcommand{\SU}{\mathrm{SU}}
\newcommand{\U}{\operatorname{U}}
\newcommand{\ST}{\mathrm{ST}}

\newcommand{\Nm}{\operatorname{Nm}}

\newcommand{\Symm}{\mathrm{Symm}}

\newcommand{\Char}{\mathrm{Char}}

\usepackage{url}

\newcommand{\ra}{\rightarrow}

\begin{document}
\title[Potentially $\GL_2$-type abelian varieties and the Sato--Tate conjecture]{Tate module tensor decompositions and the Sato--Tate conjecture for certain abelian varieties potentially of $\GL_2$-type}
\author{Francesc Fit\'e}

\address{Departament de Matemàtiques i Informàtica\\
Universitat de Barcelona\\Gran via de les Corts Catalanes, 585\\
08007 Barcelona, Catalonia}
\email{ffite@ub.edu}
\urladdr{https://www.ub.edu/nt/ffite/}

\author{Xavier Guitart}
\address{Departament de Matemàtiques i Informàtica\\
Universitat de Barcelona\\Gran via de les Corts Catalanes, 585\\
08007 Barcelona, Catalonia
}
\email{xevi.guitart@gmail.com}
\urladdr{http://www.maia.ub.es/~guitart/}
\date{\today}

\begin{abstract} We introduce a tensor decomposition of the $\ell$-adic Tate module of an abelian variety $A_0$ defined over a number field which is geometrically isotypic. If $A_0$ is potentially of $\GL_2$-type and defined over a totally real number field, we use this decomposition to describe its Sato--Tate group and to prove the Sato--Tate conjecture in certain cases. 
\end{abstract}
\maketitle
\tableofcontents

\section{Introduction}

Let $A_0$ be an abelian variety defined over a number field $k_0$ of dimension $g\geq 1$. Throughout the article, we assume that $A_0$ is \emph{geometrically isotypic}. By this we mean that the base change $A_{0,\Qbar}=A_0\times_{k_0}\Qbar$ of $A_0$ to an algebraic closure $\Qbar$ of~$\Q$ is isogenous to the power of a simple abelian variety $B$ defined over~$\Qbar$, say
$$
A_{0,\Qbar}\sim B^d\,,\qquad \text{for some $d\geq 1$}\,. 
$$
The main novelty of this article is the description of the rational $\ell$-adic Tate module $V_\ell(A_0)$ attached to $A_0$ as (the induction of) a tensor product of an Artin representation and a system of $\ell$-adic representations. To achieve such a description, we may reduce to the case that $A_0$ is simple (over $k_0$), assumption which we will make from now on. 

Under these assumptions, the endomorphism algebra of $B$, denoted $\End(B)$, is a central simple algebra over a number field $M$, whose degree we denote by~$m$. Let~$n$ be the Schur index of $\End(B)$. This means that $\End(B)$ is a matrix algebra over an $M$-central division algebra $D$ such that $[D:M]=n^2$. Let $G_{k_0}$ denote the absolute Galois group of~$k_0$. The next result accounts for the essential statements of Theorem \ref{theorem: inddec}.

\begin{theorem}\label{theorem: 1} Let $A_0$ be a simple abelian variety defined over a number field $k_0$ of dimension $g\geq 1$ and geometrically isotypic, and let $\ell$ be a prime. Then there exist:
\begin{enumerate}[i)]
\item a finite Galois extension $k/k_0$ of degree dividing $m$;
\item a number field $F$;
\item primes $\lambda_1,\dots,\lambda_r$ of $F$ lying over $\ell$, where $r=m/[k:k_0]$;
\item $F$-rational $\lambda_i$-adic representations $V_{\lambda_i}(B)^{\alpha_B}$ of $G_k$ of $\Qbar_\ell$-dimension $2g/(dnm)$, for $1\leq i\leq r$; and
\item a finite image representation $V(B,A)^{\alpha_B}$ of $G_k$ realizable over $F$;
\end{enumerate}
such that there is an isomorphism
$$
V_\ell(A_0)\otimes \Qbar_\ell \simeq \Ind_{k_0}^k\left(\bigoplus_{i=1}^{r} V_{\lambda_i}(B)^{\alpha_B}\otimes_{\Qbar_\ell} V_{\lambda_i}(B,A)^{\alpha_B}\right)
$$
of $\Qbar_\ell[G_{k_0}]$-modules. In the above isomorphism, $V_{\lambda_i}(B,A)^{\alpha_B}$ denotes the tensor product $V(B,A)^{\alpha_B}\otimes_{F,\sigma_i} \Qbar_\ell$ taken with respect to the embedding $\sigma_i\colon F\hookrightarrow \Qbar_\ell$ corresponding to the prime $\lambda_i$.
\end{theorem}

This theorem is proven in the course of Section~\ref{section: td}. Along the way we describe the number fields $F$ and $k$, and construct the Artin representation $V(B,A)^{\alpha_B}$ and the $\lambda_i$-adic representations $V_{\lambda_i}(B)^{\alpha_B}$. The representations $V_{\lambda_i}(B)^{\alpha_B}$ and $V(B,A)^{\alpha_B}$ arise naturally as projective representations. The obstruction for these projective representations to lift to genuine representations is given by two cohomology classes in $H^2(G_k,M^{\times})$ that, by a theorem of Tate, can be trivialized after enlarging the field of coefficients. It lies at the core of the proof of Theorem~\ref{theorem: 1} the fact that these two cohomology classes are inverses to each other. 

There are at least two known particular cases of this decomposition in the literature. On the one hand, it is known when~$A_0$ is $\Qbar$-isogenous to the power of an elliptic curve $B$ defined over $\Qbar$ which admits a model up to $\Qbar$-isogeny \emph{defined over~$k_0$} (this follows from \cite[Thm. 3.1]{Fit13} when $B$ does not have CM and from \cite[(3-8)]{FS14} when it does). We note that if $g$ is odd, then there does exist a  model up to $\Qbar$-isogeny for $B$ defined over $k_0$ (see Remark \ref{remark: instance}), but this is not always satisfied when $g$ is even as shown in \cite[\S3D]{FS14}. On the other hand, an analogous tensor decomposition has been obtained by N. Taylor \cite[\S3.3]{Tay19} when $A_0$ is an abelian surface with QM; see also \cite[Prop. 9.2.1]{BCGP18}. Taylor's explicit, but intriguing to us, construction of the tensor decomposition in the case of a QM abelian surface was a source of inspiration for the present article. Section \ref{section: td} is our attempt to give a uniform, general, and more conceptual explanation of this phenomenon.

The second main contribution of the article is an application of the obtained description of the Tate module of~ $A_0$ in the context of the Sato--Tate conjecture for abelian varieties potentially of $\GL_2$-type. Following Ribet, we say that $A_0$ is of $\GL_2$-type if there exists a number field of degree $g$ that injects into $\End(A_{0})$. In this article, by saying that $A_0$ is \emph{potentially of $\GL_2$-type}, we mean that there exists a number field of degree $g$ that injects into $\End(A_{0,\Qbar})$. In this case, the absolutely simple factor $B$ is often referred to as a ``building block''. If $k_0$ is totally real, three mutually excluding situations arise:
\begin{enumerate}
\item[(CM)] $A_0$ has potential complex multiplication, that is, there exists a number field of degree $2g$ that injects into $\End(A_{0,\Qbar})$.
\item[(RM)] $\End(B)$ is a totally real field of degree $\dim(B)$, in which case we say that~$B$ has real multiplication.
\item[(QM)] $\End(B)$ is a quaternion division algebra over a totally real field of degree $\dim(B)/2$ in which case we say that $B$ has quaternionic multiplication. 
\end{enumerate}
Ribet \cite{Rib92} gave proofs of these facts when $A_0$ is defined over $\Q$ (see \cite[Thm. 3.3, Prop. 3.4]{Gui12} for proofs in the general case). 

Ribet extensively studied abelian varieties of $\GL_2$-type over $\Q$ and showed that they share many features with elliptic curves. In particular, Ribet \cite[\S3]{Rib92} showed that one can attach a rank 2 compatible system of $\ell$-adic representations to an abelian variety of $\GL_2$-type. This was extended by Wu \cite{Wu18} to abelian varieties potentially of $\GL_2$-type defined over a totally real number field. 

From Section~\ref{section: GL2} on, we assume that both $k_0$ and the field $k$ produced by Theorem~\ref{theorem: 1} are totally real, and that $A_0$ is potentially of $\GL_2$-type and does not have potential CM (this already implies that $A_0$ is geometrically isotypic). In this case, we can use the results of Wu to show that the representations $V_{\lambda_i}(B)^{\alpha_B}$ are part of a rank 2 compatible system $(V_{\lambda}(B)^{\alpha_B})_\lambda$ with certain desirable properties. 

The Sato--Tate conjecture is an equidistribution statement regarding the Frobenius conjugacy classes acting on $V_\ell(A_0)$. The distribution of these classes is predicted to be governed by a compact real Lie group $\ST(A_0)$, called the Sato--Tate group of $A_0$. As shown by Serre \cite{Ser89}, this equidistribution is implied by the conjectural analytic behavior of partial Euler products attached to the irreducible representations of $\ST(A_0)$. In Section~\ref{section: STgroups}, we use Theorem \ref{theorem: 1} to determine $\ST(A_0)$. 
In Section~\ref{section: scenarios}, using deep potential automorphy results (covered by \cite{BLGGT14}) relative to the compatible system $(V_\lambda(B)^{\alpha_B})_\lambda$, we prove the desired analytic properties of the relevant partial Euler products in certain situations.

These situations are described by the following theorem. Let $K_0/k_0$ denote the minimal extension over which all the endomorphisms of $A_0$ are defined.

\begin{theorem}\label{theorem: 2}
Suppose that $k_0$ is a totally real and that $A_0$ is an abelian variety defined over $k_0$ of dimension $g\geq 1$ which is $\Qbar$-isogenous to the power of an abelian variety $B$ which is either:
\begin{enumerate}[i)]
\item an elliptic curve; or
\item an abelian surface with QM; or
\item an abelian surface with RM; or
\item an abelian fourfold with QM.
\end{enumerate}
If the field extension $k/k_0$ from Theorem \ref{theorem: 1} is trivial and $K_0/k_0$ is solvable, then the Sato--Tate conjecture holds for $A_0$.
\end{theorem}

In the above theorem, the condition of $B$ falling in one of the cases $i),\dots,iv)$ amounts to requiring that the center $M$ of the endomorphism algebra of $B$ be a number field of degree $m\leq 2$. This constraint on the degree $m$ ensures the applicability of results of Shahidi \cite{Sha81} on the invertibility of the Rankin-Selberg product of  automorphic $L$-functions, which are essential to the proof.

Let us make two remarks on the hypotheses of the theorem. On the one hand, the hypothesis that $k/k_0$ be trivial is vacuous for $i)$ and $ii)$. On the other hand, one can dispense with the hypothesis that $K_0/k_0$ be solvable when $g\leq 3$. For $g=2$, the extension $K_0/k_0$ is in fact known to be always solvable (as a byproduct of the classification in \cite{FKRS12}) and for $g=3$ it can only fail to be solvable when~$B$ is an elliptic with CM (as follows from the classification of \cite{FKS20}), in which case the theorem is known to hold as well (see Remark \ref{remark: CMB}).

Some particular instances of Theorem~\ref{theorem: 2} are known. Indeed, the works of Johansson \cite{Joh17} and N. Taylor \cite{Tay19} altogether imply the theorem when $g=2$; in other words, when $A_0$ is $\Qbar$-isogenous to the square of an elliptic curve, to an RM abelian surface, or to a QM abelian surface. Their proof is based on a case-by-case analysis using the classification of Sato--Tate groups of abelian surfaces defined over totally real number fields\footnote{Of the 35 possibilities for the Sato--Tate group of an abelian surface defined over a totally real field, 28 occur only among abelian surfaces which are geometrically isotypic and potentially of $\GL_2$-type. It should be noted that the works of Johansson and N. Taylor yield the Sato--Tate conjecture in 5 non geometrically isotypic cases as well: indeed, they yield the Sato--Tate conjecture in all but 2 of the 35 possible cases.} (as achieved in \cite{FKRS12}). Our proof of Theorem \ref{theorem: 2} is indebted to \cite{Joh17} and \cite{Tay19} in many aspects.

A second situation where Theorem~\ref{theorem: 2} was essentially known is when $g\leq 3$ and~$B$ is an elliptic curve with CM that admits a model up isogeny defined over $k_0$. Indeed, the computation of the moments of the measure governing the equidistribution of the normalized Frobenius traces of~$A_0$ in this situation was carried out in \cite[\S3]{FS14} (for $g=2$) and in \cite[\S2]{FLS18} (for $g=3$). 

Modular abelian varieties are a natural source of geometrically isotypic abelian varieties of $\GL_2$-type defined over $\Q$. Restricted to this setting, the hypotheses on the extensions $K_0/k_0$ and $k/k_0$ are automatically satisfied and Theorem \ref{theorem: 2} can be presented in the following way.

\begin{corollary}
Let $f=\sum a_m q^m\in S_2(\Gamma_1(N))$ be a newform of Nebentypus~$\varepsilon$. Let $A_f$ denote the abelian variety defined over $\Q$ associated to $f$ by the Eichler--Shimura construction. If $f$ is non-CM, suppose that the field $\Q(\{a_m^2/\varepsilon(m)\}_{(m,N)=1})$ has degree at most $2$ over $\Q$. Then the Sato--Tate conjecture holds for $A_f$. 
\end{corollary}

As we will later discuss, we note that there exist modular abelian varieties $A_f$ of arbitrarily large dimension for which the field  $\Q(\{a_m^2/\varepsilon(m)\}_{(m,N)=1})$ has degree at most~$2$.

\subsection*{Conventions and notations.} Throughout the article $k_0$ is a number field and all of its algebraic field extensions are assumed to be contained in a fixed algebraic closure $\Qbar$ of $\Q$. For each prime $\ell$, we fix an algebraic closure of $\Q_\ell$ and all finite extensions of $\Q_\ell$ are assumed to be contained in this fixed algebraic closure. We work in the category of abelian varieties up to isogeny. In particular, isogenies become invertible and $\Hom(C,D)$, for a pair abelian varieties $C$ and $D$ defined over $k_0$, is equipped with a $\Q$-vector space structure. Given a field extension $k/k_0$, we write $C_k$ to denote the base change $C\times_{k_0} k$ of $C$ from $k_0$ to $k$. We refer to nonzero prime ideals of the ring of integers of a number field $E$ simply by primes of $E$. We denote by $I$ the identity matrix, whose size should always be clear from the context.

\subsection*{Acknowledgements} Thanks to Andrew Sutherland and John Voight for suggesting the existence of a theorem of the type of Theorem~\ref{theorem: inddec} (at least in the case of a geometric power of an elliptic curve) in a conversation during the conference ``Arithmetic geometry, Number Theory, and Computation", held at MIT in August 2018.  Thanks to the referee for valuable suggestions for the revision and improvement of the article. This material is based upon work supported by the National Science Foundation under Grant No. DMS-1638352.  Fit\'e was additionally supported by the Simons Foundation grant 550033 and by MTM2015-63829-P.  Guitart  was  partially  supported  by  projects  MTM2015-66716-P and MTM2015-63829-P. This project has received funding from the European Research Council (ERC) under the European Union's Horizon 2020 research and innovation programme (grant agreement No 682152). Fit\'e expresses deep gratitude to IAS for the excellent working conditions offered during the academic year 2018-19.

\section{A Tate module tensor decomposition}\label{section: td}

Throughout this section $A_0$ will denote a simple abelian variety of dimension $g\geq 1$ defined over the number field~$k_0$ such that:

\begin{hypothesis}\label{hypothesis: main}
 $A_0$ is \emph{geometrically isotypic}, that is, $A_{0,\Qbar} \sim B^d$, where $B$ is a simple abelian variety defined over $\Qbar$ and $d\geq 1$.
\end{hypothesis}

By the Wedderburn theorem, we know that $\End(B)$ is a central division algebra over a number field $M$. Let $n$ denote the Schur index of $\End(B)$ and $m=[M\colon\Q]$.

Let $K_0/k_0$ denote the endomorphism field of $A_0$, that is, the minimal extension of $k_0$ such that 
$$
\End(A_{0,K_0})\simeq \End(A_{0,\Qbar})\,.
$$ 
The extension $K_0/k_0$ is Galois and finite. Let $K/k_0$ denote a finite Galois extension containing $K_0/k_0$.
Without loss of generality we may assume that $B$ is defined over $K$ and that
$$
\Hom(B_K,A_{0,K})\simeq \Hom(B_\Qbar,A_{0,\Qbar})\,.
$$ 

\begin{definition}
  For a subextension $k/k_0$ of $K/k_0$, the abelian variety $B$ is called a $k$-abelian variety (or $k$-variety for short) if for every $s\in \Gal(K/k)$ there exists an isogeny $\mu_s \colon {}^s B\ra B$ compatible with the endomorphisms of $B$, that is, such that
  \begin{align}\label{eq: compatibility}
  \mu_s \circ {}^s \varphi = \varphi\circ \mu_s\text{ for all }\varphi\in \End(B).    
  \end{align}

\end{definition}

Let $k/k_0$ be a subextension of $K/k_0$ such that $B$ is a $k$-variety (such a  subextension obviously exist). From now on, write $A$ for the base change $A_0\times_{k_0}k$.
Fix a system of isogenies $\{\mu_s\}_{s\in \Gal(K/k)}$ compatible with $\End(B)$ in the sense of \eqref{eq: compatibility}. We can, and do, assume that $\mu_s$ is the identity for every $s\in G_K$. 

If we equip $M^\times$ with the trivial action of $\Gal(K/k)$, the map
$$
c_B\colon \Gal(K/k)\times \Gal(K/k)\rightarrow M^ \times\,,\qquad c_B(s,t)=\mu_{st}\circ \acc s \mu_t^ {-1}\circ\mu_{s}^{-1}\,,
$$
satisfies the $2$-cocycle condition and defines a cohomology class $\gamma_B\in H^2(K/k,M^{\times})$. The cocycle $c_B$ (resp. the cohomology class $\gamma_B$) gives rise by inflation to a continuous cocycle in $Z^2(G_k,M^\times)$ (resp. a cohomology class in $H^2(G_k,M^\times)$) that we will also denote by $c_B$ (resp. $\gamma_B$).

\begin{lemma}\label{lemma: alpha}
There is a continuous map
$$
\alpha_B\colon G_k\rightarrow \Qbar^\times
$$
such that for every $s,t \in G_k$ we have 
\begin{equation}\label{equation: 1boundaryfin}
c_B(s,t)=\frac{\alpha_B(s)\alpha_B(t)}{\alpha_B(st)}\,.
\end{equation}
\end{lemma}

\begin{proof} 
The lemma is a consequence of a theorem of Tate (see \cite[Thm. 6.3]{Rib92}), which states that $H^2(G_k,\Qbar^\times)$ is trivial, where $\Qbar^\times$ is endowed with the trivial action of $G_k$. 
\end{proof}

Let $E$ denote a maximal subfield contained in $\End(B)$. The field $E$ contains the center with degree $[E:M]=n$. Since~$\alpha_B$ is continuous we may enlarge $K$ so that~$\alpha_B$ is trivial when restricted to $G_K$, and we will do this from now on. Let $F$ denote a number field containing $E$ and the values of $\alpha_B$.

Fix a rational prime~$\ell$ and an embedding $\sigma\colon F\rightarrow \Qbar_\ell$.
Denote by $\lambda=\lambda(\sigma)$ the prime of $F$ above $\ell$ for which $\sigma$ factors via the natural inclusion of $F$ into its completion $F_\lambda$ at $\lambda$.

Let $V_\ell(A)$ (resp. $V_\ell(B)$) denote the rational Tate module of $A$ (resp. $B$). For a prime $\lambda$ of $F$ above $\ell$, use the natural $E$-module structure of $V_\ell(B)$ to define
\begin{align}\label{eq: Vsigma}
V_\lambda(B)=V_\ell(B)\otimes_{E\otimes \Q_\ell,\sigma} \Qbar_\ell\,.
\end{align}
Here the tensor product is taken with respect to the map induced by the inclusions  $E\subseteq F$ and $\sigma\colon F\hookrightarrow \Qbar_\ell$. The module $V_\lambda(B)$ has dimension 
$$
\dim_{\Qbar_\ell}(V_\lambda(B))=
\frac{2\dim B}{[E\colon \Q]}= \frac{2g}{dnm}\,.
$$
It is endowed with an action of $G_k$ by the next lemma. 

\begin{lemma}
The map
$$
\varrho_\lambda^{\alpha_B}\colon G_k\rightarrow \GL(V_\lambda(B))\,,
$$
defined for $s \in G_k$ as the composition
\begin{equation}\label{equation: comp1}
\varrho_\lambda^{\alpha_B}(s)\colon V_\lambda(B)\xrightarrow {s\otimes 1} V_\lambda(\acc{s}B) \xrightarrow{\mu_{s,*}} V_\lambda(B) \xrightarrow {1\otimes \sigma(\alpha_B(s))} V_\lambda(B)\,,
\end{equation}
is a continuous representation. Here, $\mu_{s,*}$ denotes the isomorphism of Tate modules induced by the isogeny $\mu_s\colon \acc s B \rightarrow B$. 
\end{lemma}

\begin{proof}
We first check that the action is indeed $\Qbar_\ell$-linear. This amounts to note that, in virtue of \eqref{eq: compatibility}, for every $s\in G_k$, $v\in V_\ell(B)$, and $\varphi \in E$ we have
$$
\varrho_\lambda^{\alpha_B}(s)(\varphi(v)\otimes 1)= \mu_{s,*}\acc s\varphi_{*}(\acc s v)\otimes \sigma(\alpha_B(s))=\varphi_*\mu_{s,*}(\acc sv)\otimes \sigma(\alpha_B(s))=\varphi\varrho_\lambda^{\alpha_B}(s)(v\otimes 1)\,.
$$
The proof is then a straightforward computation based on \eqref{equation: 1boundaryfin}. Indeed, for every $s,t\in G_k$ and $v \in V_\ell(B)$ we have:
$$
\begin{array}{lll}
\varrho_\lambda^{\alpha_B}(st)(v\otimes 1) & = & \mu_{st,*}({}^{st}v)\otimes \sigma(\alpha_B(st))\\[4pt] 
& = & \mu_{s,*}{}^s\mu_{t,*}({}^{st}v)\cdot c_B(s,t)\otimes \sigma(\alpha_B(st)) \\[4pt]
& = & \mu_{s,*}{}^s\mu_{t,*}({}^{st}v)\otimes \sigma(\alpha_B(s)\alpha_B(t)) \\[4pt]
& = & \varrho_\lambda^{\alpha_B}(s)(\varrho_\lambda^{\alpha_B}(t)(v\otimes 1))\,.
\end{array}
$$
Since it suffices to verify continuity in a neighborhood of the identity, we are reduced to show that $\varrho_\lambda^{\alpha_B}|_{G_K}$ is continuous. But note that the action of $G_K$ via $\varrho_\lambda^{\alpha_B}$ coincides with the natural action of $G_K$ on $V_\lambda(B)$, which is continuous. 
\end{proof}

The map 
$$
E\ra\Hom(B_K,A_K)\,,
$$ 
given by precomposition of maps, equips $\Hom(B_K,A_K)$ with an $E$-module structure, 
which we use to define
$$
V(B,A)=\Hom(B_K,A_K)\otimes_E F\,.
$$
Observe that $V(B,A)$ has dimension 
$$
\dim_{F}(V(B,A))=d\frac{\dim_{\Q} \End(B)}{[E:\Q]}=\frac{dn^2m}{nm}=
dn\,.
$$
We next equip $V(B,A)$ with an action of $\Gal(K/k)$ by means of the following lemma (compare with \cite[Lemma 2.15]{FG20}).

\begin{lemma}
The map
$$
\theta^{\alpha_B}\colon \Gal(K/k)\rightarrow \GL(V(B,A))
$$
defined for $s \in \Gal(K/k)$ as the composition
\begin{equation}\label{equation: comp2}
\theta^{\alpha_B}(s)\colon V(B,A)\xrightarrow{s\otimes 1} V(\acc s B,A)\xrightarrow{(\mu_s^{-1})^*} V(B,A)\xrightarrow{1\otimes \sigma(\alpha_B(s)^{-1})}V(B,A)
\end{equation}
is a representation. Here, $(\mu_s^{-1})^*$ is the map obtained by precomposition with~$\mu_s^{-1}$. 
\end{lemma}

\begin{proof}
We first verify that the action is $F$-linear. For every $s \in \Gal(K/k)$, $\psi\in \Hom(B_K,A_K)$, and $\varphi\in E$ we have
$$
\begin{array}{lll}
\theta^{\alpha_B}(s)(\psi \circ \varphi \otimes 1)&=&\acc s \psi \circ \acc s \varphi \circ \mu_s^{-1}\otimes \sigma(\alpha_B(s)^{-1})\\[4pt]
&=&\acc s \psi \circ \mu_s^{-1}\circ \varphi\otimes \sigma(\alpha_B(s)^{-1})\\[4pt]
&=&\theta^{\alpha_B}(s)(\psi \otimes 1) \circ \varphi\,.
\end{array}
$$
For every $s,t \in G_k$ and $v\in \Hom(B_K,A_K)$, we have that
$$
\begin{array}{lll}
\theta^{c_B}(st)(\psi\otimes 1) & = & (\mu_{st}^{-1})^*({}^{st}v)\otimes \sigma(\alpha_B(st)^{-1})\\[4pt] 
& = & (\acc s\mu_{t}^{-1})^*(\mu_s^{-1})^*({}^{st}v)\cdot c_B(s,t)^{-1}\otimes \sigma(\alpha_B(st)^{-1}) \\[4pt]
& = &  (\acc s\mu_{t}^{-1})^*(\mu_s^{-1})^*({}^{st}v) \otimes \sigma(\alpha_B(s)^{-1}\alpha_B(t)^{-1}) \\[4pt]
& = & \theta^{\alpha_B}(s)(\theta^{\alpha_B}(t)(\psi\otimes 1))\,.
\end{array}
$$
\end{proof}
For a prime $\lambda$ of $F$ above $\ell$, attached to the embedding $\sigma\colon F\hookrightarrow \Qbar_\ell$, define 
$$
V_\lambda(B,A)=V(B,A)\otimes_{F,\sigma}\Qbar_\ell\,.
$$
Let $\theta^{\alpha_B}_\lambda$ denote the representation on $V_\lambda(B,A)$ obtained by letting $\Gal(K/k)$ act trivially on $\Qbar_\ell$ and via $\theta^{\alpha_B}$ on $V(B,A)$. 
Let us write  $V_\lambda(B)^{\alpha_B}$ and $V_\lambda(B,A)^{\alpha_B}$ to denote $V_\lambda(B)$ and $V_\lambda(B,A)$ equipped with actions of $G_k$ via $\varrho_\lambda^{\alpha_B}$ and $\theta_\lambda^{\alpha_B}$, respectively. 

Since $A_\Qbar\sim B^d$, we have that $\End(A_\Qbar)\simeq \M_d(\End(B))$ and therefore the center of $\End(A_\Qbar)$ can be identified with $M$. Since $B$ is a $k$-variety, we have that $A_\Qbar$ is also a $k$-variety \cite[Proposition 3.2]{Gui10}, and this implies that $M\subseteq \End(A)$ \cite[Proposition 3.15]{Gui10}; that is to say, the endomorphisms of the center of $\End(A_\Qbar)$ are defined over $k$. Therefore, $V_\ell(A)$ can be endowed with a $M$-module structure. Define
$$
V_\lambda(A):=V_\ell(A)\otimes_{\Q_\ell\otimes M,\sigma}\Qbar_\ell\,,
$$
where the tensor product is taken with respect to the map obtained from the iclusions $M\subseteq F$ and $\sigma\colon F\hookrightarrow \Qbar_\ell$. We regard $V_\lambda(A)$ as a $\Qbar_\ell[G_k]$-module by letting~$G_{k}$ act naturally on $V_\ell(A)$ and trivially on~$\Qbar_\ell$ (this is well defined because $M \subseteq \End(A)$). Observe that $V_\lambda(A)$ has dimension
$$
\dim_{\Qbar_\ell}(V_\lambda(A))=
\frac{2g}{m} = \dim_{\Qbar_\ell}(V_\lambda(B)^{\alpha_B}) \cdot \dim_{\Qbar_\ell}(V_\lambda(B,A)^{\alpha_B})\,.
$$

\begin{proposition}\label{proposition: decomposition Tate}
There is an isomorphism of $\Qbar_\ell[G_k]$-modules
\begin{equation}\label{equation: tensordec}
V_\lambda(A) \simeq V_\lambda(B)^{\alpha_B} \otimes_{\Qbar_\ell} V_\lambda(B,A)^{\alpha_B}\,. 
\end{equation}
\end{proposition}

\begin{proof}
The module $V_\lambda(A)$ is semisimple by Faltings's theorem. Let us first assume that $V_\lambda(A)$ is an irreducible $\Qbar_\ell[G_k]$-module.
Since both $V_\lambda(A)$ and $V_\lambda(B)\otimes V_\lambda(B,A)$ have the same dimension over $\Qbar_\ell$, it will suffice to show that 
$$
W:=\Hom_{G_k}(V_\lambda(A), V_\lambda(B)^{\alpha_B}\otimes V_\lambda(B,A)^{\alpha_B})\not=0.
$$
Observe that
$$
\begin{array}{lll}
W &=&\Hom_{G_k}(V_\lambda(A) \otimes (V_\lambda(B)^{\alpha_B})^\vee, V_\lambda(B,A)^{\alpha_B})\\[6pt]
& = & \Hom_{G_k}(\Hom_{G_k}(V_\lambda(B)^{\alpha_B},V_\lambda(A)), V_\lambda(B,A)^{\alpha_B})\,.
\end{array}
$$
Thus, to show that $W\not=0$, it is enough to show that the map
$$
\Phi\colon V_\lambda(B,A)^{\alpha_B}\rightarrow \Hom_{G_k}(V_\lambda(B)^{\alpha_B},V_\lambda(A))\,,\qquad \Phi(f):=f_*
$$
is $G_k$-equivariant. But this indeed holds: 
$$
\begin{array}{lll}
 \Phi(\theta_\lambda^{\alpha_B}(s)(f))& = & (\acc s f_{*} \circ \mu_{s,*}^{-1})\otimes \sigma(\alpha(s)^{-1})\\[6pt]
&= & \acc s (f_*\circ s^{-1}\circ \mu_{s,*}^{-1})\otimes \sigma(\alpha(s)^{-1}) \\[6pt]
&= & \acc s (f_* \circ \mu_{s^{-1},*})\otimes \sigma(c_B(s,s^{-1})\cdot\alpha(s)^{-1}) \\[6pt] 
&= & \acc s (\Phi(f)\circ \varrho_\lambda^{\alpha_B}(s)^{-1})\,, 
\end{array}
$$
where we have used that $c_B(s,s^{-1})=\acc s\mu_{s^{-1}}^{-1} \mu_{s}^{-1}$ and $c_B(s,s^{-1})=\alpha(s)\alpha(s^{-1})$.
To conclude, note that if $V_\lambda(A)$ decomposes, then $V_\lambda(B,A)$ does it accordingly, and we can apply the above argument to each of the respective irreducible constituents. 
\end{proof}

Note that the proposition implies, in particular, that $V_\lambda(B)^{\alpha_B}\otimes_{\Qbar_\ell} V_{\lambda}(B,A)^{\alpha_B}$ only depends on the restriction of $\sigma\colon F\rightarrow \Qbar_\ell$ to $M$. Let $\sigma_i\colon F \rightarrow \Qbar_\ell$, for $i=1,\dots,m$, denote extensions to $F$ of the distinct embeddings of $M$ into $\Qbar_\ell$. Let $\lambda_i$ denote the prime of $F$ attached to $\sigma_i$.

\begin{proposition}\label{proposition: tdt}
There is an isomorphism of $\Qbar_\ell[G_k]$-modules
$$
V_\ell(A)\otimes_{\Q_\ell} \Qbar_\ell \simeq \bigoplus_{i=1}^{m} V_{\lambda_i}(B)^{\alpha_B}\otimes_{\Qbar_\ell} V_{\lambda_i}(B,A)^{\alpha_B}\,.
$$
\end{proposition}

\begin{proof}
This follows from the isomorphism
\begin{equation}\label{equation: decemb}
V_\ell(A)\otimes_{\Q_\ell} \Qbar_\ell\simeq \bigoplus_{i=1}^{m} V_{\lambda_i}(A)\,,
\end{equation}
together with Proposition \ref{proposition: decomposition Tate}.
\end{proof}

\begin{proposition}\label{proposition: noiso}
For $i\not=j$, we have that
$$
V_{\lambda_i}(B)^{\alpha_B}\not \simeq V_{\lambda_j}(B)^{\alpha_B}
$$
as $\Qbar_\ell[G_{K'}]$-modules for any finite extension $K'/k$.
\end{proposition}

\begin{proof}
Without loss of generality we may assume that $K\subseteq K'$. On the one hand, we then have
$$
\End_{G_{K'}}(V_\ell(A)\otimes \Qbar_\ell)\simeq \M_{nd}(\End(\bigoplus_{i}V_{\lambda_i}(B)^{\alpha_B}))\,.
$$
On the other hand, we have
$$
\End(A_{K'})\otimes \Qbar_\ell\simeq \M_{d}(\End(B)\otimes \Qbar_\ell)\simeq \M_{nd}(M\otimes \Qbar_\ell)\,. 
$$
By Faltings's isogeny theorem \cite{Fal83}, we have that $\dim_{\Qbar_\ell}(\End(\oplus_{i}V_{\lambda_i}(B)^{\alpha_B}))=m$, and the proposition follows.
\end{proof}

So far, the subextension $k/k_0$ of $K/k_0$ has only been subject to the constraint  that $B$ be a $k$-variety. We now make a specific choice of $k/k_0$ that allows for a particularly nice description of the Tate module of $A_0$ in terms of that of $A=A_0\times_{k_0}k$.

\begin{theorem}\label{theorem: inddec}
Let $A_0$ be a simple abelian variety defined over $k_0$ satisfying Hypothesis~\ref{hypothesis: main}. Let $M_0=M\cap \End(A_0)$. Then $M/M_0$ is Galois and there exists a Galois subextension $k/k_0$ of $K_0/k_0$ of degree $[M:M_0]$ such that for $A=A_0\times_{k_0} k$ the following properties hold:
\begin{enumerate}[i)]
\item $M\subseteq \End(A)$.
\item $B$ is a $k$-variety. 
\item For every rational prime $\ell$, we have
$$
V_\ell(A_0)\otimes_{\Q_\ell} \Qbar_\ell\simeq \bigoplus_{\lambda}\Ind^{k}_{k_0}\left( V_\lambda(B)^{\alpha_B} \otimes_{\Qbar_\ell} V_{\lambda}(A,B)^{\alpha_B}\right)\,,
$$
where the sum runs over the primes $\lambda=\lambda(\sigma)$ of $F$ lying over $\ell$ attached to extensions $\sigma\colon F\rightarrow \Qbar_\ell$ of the $[M_0:\Q]$ distinct embeddings of $M_0$ into $\Qbar_\ell$.
\end{enumerate}
\end{theorem}

\begin{proof}
The existence of a Galois subextension $k/k_0$ of $K_0/k_0$ of degree $[M:M_0]$ such that $M\subseteq \End(A)$ and 
$$
\bigoplus_{i=1}^{[M:M_0]} V_\ell(A_0) \simeq \Ind^{k}_{k_0}\left( V_\ell(A)\right)
$$
is \cite[Rem. 2, p. 186]{Mil72}. As it is seen in the proof, there is an injection $\Gal(k/k_0)\hookrightarrow \Aut_{M_0}(M)$, which ensures that $M/M_0$ is Galois. The fact that $M\subseteq \End(A)$ implies that for every $s \in \Gal(K/k)$ we can fix an $M$-equivariant isogeny $\mu_s \colon \acc s B \rightarrow B$ coming from the $M$-equivariant isogeny
$$
\acc s B^d\sim \acc s A_{K} = A_ {K}\sim B^d\,.
$$
The $M$-equivariant system of isogenies $\{\mu_s\}_{s \in G_{k}}$ can be modified into a $\End(B)$-equivariant system $\{\lambda_s\}_{s \in G_{k}}$, so that $B$ becomes a $k$-variety. Indeed, consider the $M$-algebra isomorphism
$$
\End(B)\rightarrow \End(B)\,,\qquad \varphi\mapsto \mu_{s}\circ \acc s \varphi \circ \mu_s^{-1}\,.
$$
The Skolem--Noether theorem shows the existence of an element $\psi\in \End(B)^\times$ such that $\mu_{s}\circ \acc s \varphi \circ \mu_ {s}^{-1}=\psi\circ \varphi \circ \psi^{-1}$. Then define $\lambda_s=\psi^{-1}\circ \mu_{s}$. The theorem then follows by applying Proposition~\ref{proposition: tdt} to $V_\ell(A)$ and using that 
$$
\Ind^{k}_{k_0}(V_\lambda(B)^{\alpha_B}\otimes V_\lambda(B,A)^{\alpha_B})\simeq \Ind^{k}_{k_0}(V_{\lambda'}(B)^{\alpha_B}\otimes V_{\lambda'}(B,A)^{\alpha_B})
$$
if $\lambda=\lambda(\sigma)$, $\lambda'=\lambda'(\sigma')$, and $\sigma$ and $\sigma'$ coincide on $M_0$. The latter follows from Lemma \ref{lemma: inductiontrace} below.
\end{proof}

\begin{lemma}\label{lemma: inductiontrace}
Let $\lambda$ be a prime of $F$ attached to the embedding $\sigma\colon F\rightarrow \Qbar_\ell$. For $s \in G_k$ we have  
$$ 
\Tr\Ind^k_{k_0}(V_\lambda(A))(s)=\Tr_{\Q_\ell\otimes \sigma(M)/\Q_\ell\otimes \sigma(M_0)}\Tr(V_\lambda(A)(s))\,.
$$
\end{lemma}

\begin{proof}
Let $\varrho_{\lambda_i}^A$ (resp. $\varrho_\ell$) be the representation afforded by $V_{\lambda_i}(A)$ (by $V_\ell(A_0)$). For $t\in G_{k_0}$ and $s\in G_k$, define
$$
\varrho_{\lambda_j}^{A,t}(s):=\varrho_{\lambda_j}^A(tst^{-1})\,,\qquad \varrho_{\ell}^{t}(s):=\varrho_{\ell}(tst^{-1})\,.
$$
Note that $\varrho_\ell^t\simeq \varrho_\ell$ and that the isomorphism class of $\varrho_{\lambda_j}^{A,t}$ only depends on the projection of $t$ into $\Gal(k/k_0)$. If $\sigma_i=\sigma_j\circ t$, where $t$ is regarded as an element of $\Gal(M/M_0)$ via the isomorphism $\Gal(k/k_0)\simeq \Gal(M/M_0)$, then by transport of structure we have
\begin{equation}\label{equation: transportofstructure}
\varrho_{\lambda_j}^{A,t}=\varrho_\ell^t\otimes_{t^{-1}(M),\sigma_i}\Qbar_\ell\simeq \varrho_\ell\otimes_{M,\sigma_i}\Qbar_\ell= \varrho_{\lambda_i}^{A}\,.
\end{equation}
Let $T$ denote a system of representatives of $\Gal(k/k_0)$. Then
$$
\Tr\Ind^k_{k_0}(\varrho_\lambda^A)(s)=\sum_{t\in T}\Tr\varrho_{\lambda}^{A,t}(s)=\Tr_{\Q_\ell\otimes \sigma(M)/\Q_\ell\otimes \sigma(M_0)}\Tr \varrho_\lambda^A(s)\,,
$$
where we used \eqref{equation: transportofstructure} in the last equality.
\end{proof}

\begin{remark}
We will be later interested in the case that $k_0$ is totally real. Note that if $[M:M_0]$ is odd, then the injection $\Gal(k/k_0)\hookrightarrow \Aut_{M_0}(M)$ forces $k$ to be totally real as well. In the case that $k_0=\Q$ and $\Aut_{M_0}(M)$ has a single element of order $2$, then $k$ is either totally real or CM (this follows from the fact that all complex conjugations of $\Gal(k/\Q)$ are conjugate).  
\end{remark}

\begin{remark}\label{remark: instance}
Let us review a particular case of Proposition~\ref{proposition: decomposition Tate} implicit in \cite{FG18}. Suppose that $A$ is $\Qbar$-isogenous to the $g$-th power of a non-CM elliptic curve $B$ and that $g$ is odd. Then, by \cite[Theorem~2.21]{FG18}, the cohomology class $\gamma_B$ of $c_B$ in $H^2(G_k,\Q^\times)$ is trivial. By Weil's descent criterion, if $\gamma_B$ is trivial, then $B$ admits a model $B^*$ up to isogeny defined over $k$. If $L^*/k$ denotes the minimal extension such that $\Hom(B_{L^*},A_{L^*})\simeq \Hom(B_\Qbar,A_\Qbar)$, then by \cite[Thm. 3.1]{Fit13} one has that
$$
V_\ell(A)\simeq V_\ell(B^*) \otimes_{\Q_\ell} \Hom(B^*_{L^*},A_{L^*})\,,  
$$
which may be regarded as a particular instance of Proposition~\ref{proposition: decomposition Tate}.
\end{remark}

\section{The weakly compatible system $V_\lambda(B)$}\label{section: GL2}

Let $A_0$ be an abelian variety defined over $k_0$ satisfying Hypothesis \ref{hypothesis: main}. Let $k/k_0$ be as in Theorem~\ref{theorem: inddec} and write $A=A_0\times_{k_0}k$. In this section we assume further the following.

\begin{hypothesis}\label{hypothesis: potGL2}\hfill
\begin{enumerate}[i)]
\item $A_0$ is potentially of $\GL_2$-type and does not have potential CM. 
\item Both $k_0$ and $k$ are totally real.
\end{enumerate}
\end{hypothesis}

Note that under this hypothesis, we have $nm=\dim(B)$, and hence the spaces $V_\lambda(B)^{\alpha_B}$ have $\Qbar_\ell$-dimension~$2$. Observe also that Hypothesis~\ref{hypothesis: potGL2} implies Hypothesis~\ref{hypothesis: main}.
 Keep the notations~$B$,~$\alpha_B$, and~$F$ of Section~\ref{section: td}. 
The goal of this section is to present $\Rr=(V_\lambda(B))_\lambda$ as a rank $2$ weakly compatible system of $\ell$-adic representations of $G_k$ defined over $F$ (see \cite[\S5.1]{BLGGT14} for the definition of weakly compatible system of $\ell$-adic representations). This will rely on classical work of Ribet and on the following result of Wu (in fact, Wu's result extends work of Ellenberg and Skinner \cite[Prop. 2.10]{ES01}, who considered the $\dim(B)=1$ case).

\begin{proposition}[Cor. 2.1.15, Prop. 2.2.1, \cite{Wu11}]\label{proposition: Wu}
Suppose that $A_0$ satisfies Hypothesis \ref{hypothesis: potGL2}. Then there exists an abelian variety $A^{\alpha_B}$ of $\GL_2$-type defined over~$k$ satisfying:
\begin{enumerate}[i)]
\item $\dim(A^{\alpha_B})=[F:\Q]$ and there exists an inclusion $F\hookrightarrow \End(A^{\alpha_B})$.
\item There is an isomorphism of $\Qbar_\ell[G_k]$-modules
$$
V_\lambda(A^{\alpha_B})\simeq V_\lambda(B)^{\alpha_B}\,,
$$
where $V_\lambda(A^{\alpha_B})$ is the tensor product $V_\ell(A^{\alpha_B})\otimes_{\Q_\ell\otimes F,\sigma}\Qbar_\ell$ taken with respect to the map induced by the embedding $\sigma\colon F\hookrightarrow \Qbar_\ell$ attached to the prime $\lambda$.
\end{enumerate}
\end{proposition}

\begin{proposition}[Ribet]\label{proposition: Ribet}
Suppose that $A_0$ satisfies Hypothesis \ref{hypothesis: potGL2}. Then $\Rr=(V_\lambda(B)^{\alpha_B})_\lambda$ is a weakly compatible system of $\ell$-adic representations of $G_k$ defined over $F$, of rank $2$, and satisfying:
\begin{enumerate}[i)] 
\item It is pure of weight 1, regular, and with Hodge--Tate weights $0$ and $1$.
\item Its determinant $\delta_\lambda:=\det(V_\lambda(B)^{\alpha_B})$ is of the form $\varepsilon_\lambda\chi_\ell$, where 
$$
\varepsilon_\lambda\colon G_k\stackrel{\varepsilon}{\rightarrow} F^\times\stackrel{\sigma}{\hookrightarrow}\Qbar_\ell^\times
$$
is a finite order character and $\chi_\ell\colon G_k\rightarrow \Q_\ell^\times$ is the $\ell$-adic cyclotomic character. 
\item It is strongly irreducible and $\End_{G_{K'}}(V_\lambda(B)^{\alpha_B})\simeq \Qbar_\ell$, for every finite extension~$K'/k$. 
\item It is totally odd, in the sense that $\delta_\lambda(\tau)=-1$ for every complex conjugation $\tau\in G_k$.
\end{enumerate}  
\end{proposition}

\begin{proof}
By Proposition \ref{proposition: Wu}, it suffices to prove the corresponding statements for $(V_\lambda(A^{\alpha_B}))_\lambda$.  When $k=\Q$, this can be found in the work of Ribet: the Hodge--Tate property and the description of the determinant follow from \cite[lem. 3.1]{Rib92}, the totally oddness is \cite[lem. 3.2]{Rib92}, and strong irreducibility amounts to \cite[lem. 3.3]{Rib92}. See \cite[\S2.2]{Wu11} for the general statements.
\end{proof}

\begin{remark}
We may also regard $(V_\lambda(B,A)^{\alpha_B})_\lambda$ as a compatible system of $\ell$-adic representations defined over $F$. Note that its tensor product with $(V_\lambda(B)^{\alpha_B})_\lambda$ equals $(V_\lambda(A))_\lambda$ which is in fact defined over $M$. 
\end{remark}

We will later make use of the following result (see \cite[Thm. 5.4.1]{BLGGT14}).

\begin{theorem}[\cite{BLGGT14}]\label{theorem: automorphy}
Suppose that $k$ is a totally real field. Then given natural numbers $e_1,\dots,e_r\geq 0$ and a finite extension $k^*/k$, there exists a totally real extension $k'/k$ such that:
\begin{enumerate}[i)]
\item $\Symm^{e_1}(\Rr|_{G_{k'}}),\dots,\Symm^{e_r}(\Rr|_{G_{k'}})$ are all automorphic;
\item $k'/k$ is linearly disjoint from $k^*$ over $k$; and
\item $k'/\Q$ is Galois.
\end{enumerate}  
\end{theorem}

\section{Sato--Tate groups and Sato--Tate conjecture}\label{section: STgroups}

Let $A_0$ be an abelian variety defined over $k_0$ satisfying Hypothesis \ref{hypothesis: potGL2}. Assume further that $k=k_0$, where $k$ is the field given by Theorem~\ref{theorem: inddec}. We will keep the notations of Section \ref{section: td}, but note that $A$ and $A_0$ (resp. $M$ and $M_0$, etc) become now synonims. 

The aim of this section is to describe the Sato--Tate group of $A$, denoted $\ST(A)$. We will describe $\ST(A)$ as the Kronecker product of $m=[M:\Q]$ copies of $\SU(2)$ and a finite group $H$ closely related to the image of~$\theta_\lambda^{\alpha_B}$. We will also state the Sato--Tate conjecture for $A$. 

\textbf{Sato--Tate groups.}
Let us start by briefly recalling the definition of $\ST(A)$, a compact real Lie subgroup of $\USp(2g)$, only well-defined up to conjugacy. Fix a prime $\ell$ and let
$$
\varrho_\ell\colon G_k\rightarrow \GL(V_\ell(A))
$$ 
be the $\ell$-adic representation attached to $A$. Denote by $G_\ell^{\Zar}(A)$  the Zariski closure of the image of $\varrho_\ell$ in the algebraic group $\GL_{V_\ell(A)}/\Q_\ell$. The compatibility of $\varrho_\ell$ with the Weil pairing, ensures that $G_\ell^{\Zar}(A)$ sits inside $\GSp_{2g}/\Q_\ell$. Let $G_\ell^{\Zar,1}(A)$ denote the kernel of the restriction to $G_\ell^{\mathrm{Zar}}(A)$ of the similitude character of $\GSp_{2g}$. Fix an embedding $\iota\colon \Qbar_\ell\rightarrow \C$ and denote by $G_\iota^1(A)$ the group of $\C$-points of the base change of $G_\ell^{\Zar,1}(A)$ from $\Q_\ell$ to $\C$ via $\iota$. The Sato--Tate group $\ST(A)$ is defined as a maximal compact subgroup of $G_\iota^1(A)$. It is expected that this definition coincides with the definition given by Banaszak--Kedlaya \cite{BK15}, which is independent of the choice of $\ell$. We refer to \cite[\S2.1]{FKRS12} for more details. 

We next define in a similar way a Sato--Tate group for the system $(V_\lambda(B)^{\alpha_B})_\lambda$ along the lines of \cite[\S7]{BLGG11}. We will formally denote it by $\ST(B^{\alpha_B})$ in order to emphasize that it is not the Sato--Tate group of the abelian variety $B$ defined over $K$. Let $G_\lambda^{\Zar}(B^{\alpha_B})$ denote the Zariski closure of the image of 
$$
\varrho_\lambda^{\alpha_B}\colon G_k\rightarrow \GL(V_\lambda(B))
$$
in the algebraic group $\GL_{V_\lambda(B)}/\Qbar_\ell$. Let $a$ denote the order of the character $\varepsilon$ introduced in Section~\ref{section: GL2}. Let $G_\lambda^{\Zar,1}(B^{\alpha_B})$ denote the kernel of the map
$$
\det{}^a\colon G_\lambda^{\Zar}(B^{\alpha_B})\rightarrow \mathbb G_m\,.
$$
Denote by $G_\iota^1(B^{\alpha_B})$ the group of $\C$-points of the base change of $G_\lambda^{\Zar,1}(B^{\alpha_B})$ from $\Qbar_\ell$ to $\C$ via $\iota$. The Sato--Tate group $\ST(B^{\alpha_B})$ is defined as a maximal compact subgroup of $G_\iota^1(B^{\alpha_B})$.

\begin{lemma}\label{lemma: STB}
We have $\ST(B^{\alpha_B})\simeq \SU(2) \otimes \mu_{2a}$.
\end{lemma}

\begin{proof}
By the definition of $\ST(B^{\alpha_B})$, we clearly have a monomorphism
$$
\ST(B^{\alpha_B})\hookrightarrow \U(2)_a\,,
$$
where $\U(2)_a$ is the subgroup of $\U(2)$ consisting of those matrices $g\in \U (2)$ with $\det(g)^a= 1$. We may compose this monomorphism with the inverse of the group isomorphism 
$$
\SU(2) \otimes \mu_{2a}\rightarrow \U(2)_a\,,\qquad A \otimes \zeta \mapsto A\zeta
$$ 
to get a monomorphism $\varphi$. Since $-I$ lies in the image of $\varphi$ it will suffice to show that the induced monomorphism
$$
\tilde \varphi\colon \ST(B^{\alpha_B})/\langle -I\rangle \rightarrow  \SU(2)\otimes \mu_{2a}/\langle -I\otimes 1\rangle
$$
is surjective. Consider now the monomorphism
$$
\ST(B^{\alpha_B})/\langle -I\rangle \stackrel{\tilde \varphi}{\rightarrow } \SU(2)\otimes \mu_{2a}/\langle -I\otimes 1\rangle \xrightarrow{\pi_1\times \pi_2}\SU(2)/\langle -I\rangle \times \mu_{2a}/\langle -1\rangle\,,
$$
where $\pi_i$ denotes the natural projection map. Let $N_i$ denote the kernel of $\pi_i \circ \tilde \varphi$.
 By part $iii)$ (resp. part $ii)$) of Proposition \ref{proposition: Ribet}, we have that that $\pi_1 \circ \tilde \varphi$ (resp. $\pi_2\circ \tilde \varphi$) is surjective. Then by Goursat's lemma (as in \cite[lem. 5.2.1]{Rib76}) we have that
$$
\SU(2)/\tilde \varphi( N_2) \simeq \mu_{2a}/\tilde \varphi( N_1)\,.
$$ 
Since $\SU(2)$ has no proper normal subgroups of finite index, we deduce that $\tilde\varphi(N_2)\simeq \SU(2)/\langle -I\rangle$. This immediately implies that $\tilde \varphi$ is surjective.
\end{proof}

By enlarging $F$ if necessary we can assume that it contains the values $\sqrt{\varepsilon(s)}$ for $s\in G_k$, where $\varepsilon$ is the character appearing in Proposition \ref{proposition: Ribet}.  Consider the well-defined group homomorphism 
$$
\tilde\varepsilon^{1/2}\colon G_k\rightarrow F^{\times}/\langle -1\rangle\,,\qquad \tilde\varepsilon^{1/2}(s)=\sqrt{\varepsilon(s)}\pmod{\langle -1\rangle}\,,
$$ and denote by $\tilde\varepsilon^{-1/2}$ its inverse.  We will denote by $k_\varepsilon/k$ the field extension cut out by the homomorphism $\tilde\varepsilon^{1/2}$, which coincides with the field cut out by $\varepsilon$. Let us denote by 
$$
\tilde \theta_\lambda^{\alpha_B} \colon G_{k}\rightarrow \GL(V_\lambda(B,A))/\langle -I\rangle
$$
the group homomorphism naturally induced by $\theta_\lambda^{\alpha_B}$.

\begin{proposition}\label{proposition: K0} The following field extensions coincide:
\begin{enumerate}[i)]
\item The endomorphism field $K_0/k$.
\item The field extension cut out by the representation $\theta_\lambda^{\alpha_B}\otimes \theta_\lambda^{\alpha_B,\vee}$. 
\item The field extension cut out by the group homomorphism
$$
\tilde \varepsilon^{1/2}\otimes \tilde \theta_\lambda^{\alpha_B} \colon G_{k}\rightarrow \GL(V_\lambda(B,A))/\langle -I\rangle\,. 
$$
\end{enumerate}
\end{proposition}

\begin{proof}
By Faltings isogeny theorem, as in the proof of Proposition \ref{proposition: noiso}, we have that $K_0/k$ is the minimal extension of $k$ such that
$$
\End_{G_{K_0}}(V_\lambda(A)\otimes \Qbar_\ell)\simeq \M_{nd}(\Qbar_\ell)\,.
$$
Let $K'/k$ be an arbitrary finite extension. By Proposition \ref{proposition: decomposition Tate}, we have
$$
\begin{array}{lll}
\End_{G_{K'}}(V_\lambda(A)) & \simeq & \displaystyle{\End_{G_{K'}}(V_{\lambda}(B)^{\alpha_B})\otimes V_{\lambda}(B,A)^{\alpha_B} )}\\[6pt]
& \simeq & \displaystyle{\Hom_{G_{K'}}(V_{\lambda}(B)^{\alpha_B}\otimes V_{\lambda}(B)^{\alpha_B,\vee}, V_{\lambda}(B,A)^{\alpha_B}\otimes V_{\lambda}(B,A)^{\alpha_B,\vee} )}\\[6pt] 
& \simeq & \displaystyle{\left(V_{\lambda}(B,A)^{\alpha_B} \otimes V_{\lambda}(B,A)^{\alpha_B,\vee}\right)^{G_{K'}}}\,, 
\end{array}
$$
where in the last isomorphism we have used that $\End_{G_{K'}}(V_{\lambda}(B)^{\alpha_B})\simeq \Qbar_\ell$, as stated in part $ii)$ of Proposition \ref{proposition: Ribet}. This shows that the field extensions of~$i)$ and~$ii)$ coincide. In fact, we could have alternatively shown the equivalence between $i)$ and~$ii)$, by establishing the isomorphism
$$
V_\lambda(B,A)^{\alpha_B}\otimes V_\lambda(B,A)^{\alpha_B,\vee}\simeq \End(A_{K_0})\otimes_{M,\sigma} \Qbar_\ell
$$
of $\Qbar_\ell[G_k]$-modules (in the same lines as in the proof of Proposition \ref{proposition: decomposition Tate}).

Let $L$ denote the field extension cut out by $\tilde\varepsilon^{1/2}\otimes\tilde\theta_\lambda^{\alpha_B}$. We first show that $K_0\subseteq L$. Indeed, for every $s \in G_L$, we have that $\theta_\lambda^{\alpha_B}(s)$ is a scalar diagonal matrix. Thus $\theta_\lambda^{\alpha_B}\otimes \theta_\lambda^{\alpha_B,\vee}(s)$ is trivial, and then by $ii)$ we deduce that $s\in G_{K_0}$.

We will give two different proofs of the fact that $L\subseteq K_0$. 
For $s\in G_k$, let $d(\mu_s)$ denote the ``degree" of $\mu_s$ as defined on \cite[p. 223]{Pyl02}. As shown in \cite[Thm. 5.12]{Pyl02}, for $s\in G_k$, we have that\footnote{Beware that $c_B$ is the inverse of the $2$-cocycle chosen by Pyle.}
$$
\varepsilon(s)=\frac{\alpha_B(s)^2}{d(s)}\,.
$$
Let now $\varphi \in V_\lambda(B,A)$ and $s\in G_{K_0}$. Since $\mu_s$ is the identity and $d(s)=1$, we find that
$$
\tilde \varepsilon^{1/2}\otimes\tilde \theta_\lambda^{\alpha_B}(s)(\varphi)=\alpha_B(s)\cdot \acc s\varphi\circ \mu_s^{-1}\otimes \alpha_B(s)^{-1}=\varphi\,,
$$ 
which gives the first proof of the fact that $G_{K_0}\subseteq G_L$.

As for the second proof, let $s\in G_{K_0}$ so that $\theta_\lambda^{\alpha_B}(s)$ is a scalar matrix. We claim that $\tilde\theta_\lambda^{\alpha_B}(s)$ and $\tilde\varepsilon^{-1/2}(s)$ coincide as elements in $F^\times/\langle -1\rangle$. 
 
By the Chebotarev density theorem it is enough to show the claim when $s$ is of the form $\Frob_\p$, for some prime $\p$ of $k$ of good reduction for $A$. 
To shorten notation let us write
$$
a_\p=\Tr(V_{\lambda}(B)^{\alpha_B}(\Frob_\p)),\,\, b_\p=\Tr(V_{\lambda}(B,A)^{\alpha_B}(\Frob_\p)),\,\, c_\p=\Tr(V_{\lambda}(A))(\Frob_\p)\,.
$$
To prove the claim we may even restrict to primes $\p$ for which $a_\p$ is nonzero, since the density of those for which $a_\p=0$ is zero (this may be seen by applying the argument of \cite[Ex. 2, p. IV-13]{Ser89} to $V_{\lambda}(B)^{\alpha_B}$; see also \cite{Ser81}).
Recall that by \cite[Thm. 5.3]{Rib92} (see also \cite[Prop. 2.2.14]{Wu11}), we have that
\begin{equation}\label{equation: Ribagain}
\frac{a_\p^2}{\varepsilon_\p}=a_\p\overline a_\p \in M\,,
\end{equation}
where $\varepsilon_\p:=\varepsilon(\Frob_\p)$ and $\overline \cdot$ denotes the ``complex conjugation" in $F$.
By Theorem~\ref{theorem: inddec}, we have that $a_\p b_\p=c_\p\in M$. From this and \eqref{equation: Ribagain}, we see that
$$
b_\p^2\varepsilon_\p=\frac{c_\p^2\varepsilon_\p}{a_\p^2}=\frac{c_\p^2}{a_\p\overline a_\p}
$$
is a totally positive element of the totally real field $M$. The assumption that $\Frob_\p\in G_{K_0}$ implies that $b_\p=nd\zeta_\p$ for some root of unity $\zeta_\p$. We deduce that $\zeta_\p^2\varepsilon_\p=1$. This shows that $\tilde\theta_{\lambda_i}^{\alpha_B}(\Frob_\p)$ and $\tilde\varepsilon^{-1/2}(\Frob_\p)$ coincide as elements in $F^\times/\langle -1\rangle$ and the second proof of the inclusion $L\subseteq K_0$ is complete.
\end{proof}

\begin{proposition}
The field cut out by the representation
$$
\theta_\lambda^{\alpha_B} \colon G_{k}\rightarrow \GL(V_\lambda(B,A)) 
$$
is an extension  of degree at most 2 of $k_\varepsilon K_0/k$.
\end{proposition}

\begin{proof}
It will suffice to show that the field extension $L'/k$ cut out by $\tilde \theta_\lambda^{\alpha_B}$ is $k_\varepsilon K_0$.  We have that
$$
\tilde\theta_\lambda^{\alpha_B}\simeq \tilde\varepsilon^{-1/2}\otimes(\tilde\varepsilon^{1/2}\otimes \tilde\theta_\lambda^{\alpha_B})\,.
$$
First note that $K_0 \subseteq L'$. Then, by Proposition \ref{proposition: K0}, we have that $L'/K_0$ is the minimal extension cut out by $\tilde\varepsilon^{-1/2}|_{G_{K_0}}$. The proposition now follows from the fact that $k_\varepsilon$ is also the field cut out by $\tilde\varepsilon^{-1/2}$ 
\end{proof}

\begin{definition}\label{definition: H} Let $\tilde H$ denote the (isomorphic) image of the Galois group $\Gal(K_0/k)$ by the representation $\tilde\varepsilon^{1/2}\otimes \tilde\theta_\lambda^{\alpha_B}$. We will denote by $H$ the preimage of $\tilde H$ by the projection map 
$$
\GL(V_\lambda(B,A))\rightarrow \GL(V_\lambda(B,A))/\langle -I\rangle\,.
$$
\end{definition}
Recall the embeddings $\sigma_i\colon F\rightarrow \Qbar_\ell$, for $i=1,\dots,m$, obtained as extensions to~$F$ of the distinct embeddings of $M$ into $\Qbar_\ell$. They define primes $\lambda_i$ of $F$. 
Write $\varepsilon_i$ for $\iota \circ \varepsilon_{\lambda_i}$ and $\theta_i^{\alpha_B}$ for $\iota\circ \theta_{\lambda_i}^{\alpha_B}$. Let $\varepsilon_i^{1/2}$ denote an arbitrary square root of $\varepsilon_i$. Note that the map $\varepsilon_i^{1/2}$ will not be in general a character. We set
$$
\prod_{i= 1}^{m}\SU(2)^{(i)}\otimes H:=\left\{  (g_i \otimes (\varepsilon_i^{1/2}\otimes \theta_i^{\alpha_B})(h))_i \mid g_i \in \SU(2),\,h\in \Gal(K_0/k)\right\}\,. 
$$
Since $-I$ belongs to $\SU(2)$, this definition does not depend on the choice of the square root $\varepsilon_i^{1/2}$.

\begin{proposition}\label{proposition: STA}
Up to conjugacy, $\ST(A)$ is the subgroup
$$
\prod_{i= 1}^{m}\SU(2)^{(i)}\otimes H\subseteq \USp(2g)\,.
$$
In particular, we have:
\begin{enumerate}[i)]
\item The identity component $\ST(A)^0$ of $\ST(A)$ satisfies
$$
\ST(A)^0\simeq \ST(A_{K_0})\simeq 
\underbrace{\SU(2) \times \cdots \times \SU(2)}_{m}\,.
$$
\item The group of connected components $\pi_0(\ST(A))$ of $\ST(A)$ is isomorphic to $\Gal(K_0/k)$.
\end{enumerate}
\end{proposition}

\begin{proof}
Proposition \ref{proposition: tdt} and Lemma \ref{lemma: STB} imply that there is an injection
$$
\varphi\colon \ST(A)\hookrightarrow \prod_{i= 1}^{m}\SU(2)^{(i)}\otimes H\,.
$$
That the projection of $\varphi$ onto the $i$-th factor $\SU(2)^{(i)}\otimes H$ is surjective is again an application of Goursat's lemma (as in the proof of Lemma \ref{lemma: STB}). Since the $V_{\lambda_i}(B)^{\alpha_B}$ are strongly irreducible, the lack of surjectivity of $\varphi$ would translate into the existence of an isomorphism $V_{\lambda_i}(B^{\alpha_B})\simeq V_{\lambda_j}(B^{\alpha_B})$ as $\Qbar_\ell[G_{K'}]$-modules for some $i\not = j$ and some finite extension $K'/k$. This contradicts Proposition~\ref{proposition: noiso}.

The statement regarding the group of components is an immediate consequence of Proposition \ref{proposition: K0}.
\end{proof}

\textbf{Frobenius conjugacy classes.} Let $S$ denote a finite set of primes of $k$ containing the primes of bad reduction for $A$ and the primes of bad reduction for the variety $A^{\alpha_B}$ given by Proposition \ref{proposition: Wu}. Let $\varrho_i^{\alpha_B}$ stand for $\iota\circ \varrho_{\lambda_i}^{\alpha_B}$. For $\mathfrak{p} \not\in S$, let $x_\p$ be the conjugacy class in $\ST(A)$ of $( g_{i,\p}\otimes h_{i,\p})_i$, where 
$$
g_{i,\p}:=\Nm(\p)^{-1/2}\cdot \varepsilon_i^{-1/2}\otimes\varrho_{i}^{\alpha_B}(\Frob_\p)\,, \qquad h_{i,\p}:=\varepsilon_i^{1/2}\otimes \theta_i^{\alpha_B}(\Frob_{\p})\,,
$$
for an arbitrary choice of square root $\varepsilon_i(\Frob_\p)^{1/2}$ (note that the Kronecker product $g_{i,\p}\otimes h_{i,\p}$ does not depend on this choice). We will simply write $h_\p$ to denote $h_{1,\p}$. We have an equality of characteristic polynomials
$$
\Char(x_\p)=\Char\left(\Nm(\p)^{-1/2} \varrho_\ell(\Frob_\p)\right).
$$

\textbf{Sato--Tate conjecture.} In our specific situation, the general Sato--Tate conjecture (see \cite[\S2.1]{FKRS12}, \cite[Chap. 8]{Ser12}) takes the following form.
\begin{conjecture}[Sato--Tate conjecture for $A$]\label{conjecture: ST}
The sequence $\{x_\p\}_{\p\not \in S}$, where the primes $\p$ are ordered with respect to their absolute norm, is equidistributed on the set of conjugacy classes of $\ST(A)$ with respect to the projection on this set of the Haar measure of $\ST(A)$.
\end{conjecture}

Let $\varrho$ be an irreducible representation of $\ST(A)$. For $s\in \C$ with $\Re(s)> 1$, define the partial Euler product 
$$
L^S(\varrho, A, s)=\prod_{\p\not \in S}\det(1-\varrho(x_{\p})\Nm(\p)^{-s})^{-1}\,.
$$
By \cite[App. to Chap. I]{Ser89}, Conjecture \ref{conjecture: ST} holds if for every irreducible nontrivial representation~$\varrho$, the partial Euler product $L^S(\varrho,A,s)$ is invertible, that is, it extends to a holomorphic and nonvanishing function on a neighborhood of $\Re(s)\geq 1$.

\textbf{Irreducible representations of $\ST(A)$.} In the next section we will prove Conjecture \ref{conjecture: ST} in certain cases when $k_0=k$ and $[M:\Q]\leq 2$. Let us describe the type of partial Euler products that one finds in this setting. Let $H$ be as in Definition~\ref{definition: H}.

If $[M:\Q]=1$, then $k=k_0$, and we see from Proposition \ref{proposition: STA}, that the irreducible representations of $\ST(A)$ are of the form $\Symm^e\otimes \eta$, where $e$ is an integer $\geq 0$, $\Symm^e$ is the $e$-th symmetric power of the standard representation of $\SU(2)$ and $\eta$ is an irreducible representation of $H$ such that
\begin{align}\label{eq: eta(-1)}
  \eta(-I)=(-I)^e.
\end{align}
We then have
\begin{equation}\label{equation: Lfunctype}
L^S(\Symm^e\otimes \eta, A, s)=\prod_{\p\not \in S}\det(1-\Symm^e(g_{1,\p})\otimes \eta(h_{\p})\Nm(\p)^{-s})^{-1}\,.
\end{equation}

Suppose that $[M:\Q]=2$ and that $k=k_0$. Then the irreducible representations of $\ST(A)$ are of the form
\begin{equation}\label{equation: reptyp2}
\Symm^{e_1}\otimes \Symm^{e_2} \otimes \eta\,,
\end{equation}
where $e_1,e_2$ are integers $\geq 0$, and $\eta$ is an irreducible representation of $H$ such that $\eta(-I)=(-I)^{e_1+e_2}$.

For every $e\geq 0$, we next attach to any representation $\eta$ as above an Artin representation $\eta_e$ that will be used in \S\ref{section: scenarios} to link the partial Euler products described above to the partial Euler products of the compatible systems $(V_\lambda(B)^{\alpha_B})_\lambda$.

\begin{lemma}\label{lemma: eta_e}
Let $\eta:H\ra \GL(V)$ be a complex representation such that $\eta(-I)=(-I)^ e$.
For every $s\in G_{k}$, fix a choice $\varepsilon_i^{{1}/{2}}(s)$ of a square root of $\varepsilon_i(s)$. Then the map
  \begin{align*}
     \eta_e\colon G_{k}\ra \GL(V), \qquad  \eta_e(s) :=\varepsilon_i(s)^{-{e}/{2}}\otimes \eta(\varepsilon_i^{{1}/{2}}(s)\otimes\theta^{\alpha_B}_i (s))
  \end{align*}
  is a representation. Moreover, it factors through an extension $K_e$ of degree at most~$2$ of $K_0 k_\varepsilon$.
\end{lemma}
\begin{proof}
  For $s,t\in G_{k} $ define
  \begin{align*}
    c_{\varepsilon}(s,t):=\frac{\varepsilon_i^{{1}/{2}}(s)\varepsilon_i^{{1}/{2}}(t)}{\varepsilon_i^{{1}/{2}}(st)}\in\{\pm1\}.
  \end{align*}
  Then
  \begin{align*}
     \eta_e (st) &=\varepsilon_i(st)^{-{e}/{2}}\otimes \eta(\varepsilon_i^{{1}/{2}}(st)\otimes\theta^{\alpha_B}_i (st)) \\ &=
                                                                                                                                  c_\varepsilon(s,t)^{e}\varepsilon_i(s)^{-{e}/{2}}\varepsilon_i(t)^{-{e}/{2}}\otimes \eta(c_\varepsilon(s,t)\varepsilon_i^{{1}/{2}}(s)\varepsilon_i^{{1}/{2}}(t)\otimes\theta^{\alpha_B}_i (s)\theta^{\alpha_B}_i(t))\\
    &=  \eta_e(s) \eta_e(t);
  \end{align*}
  here we have used that $\eta(c_\varepsilon(s,t))=c_\varepsilon(s,t)^e$, which follows from the hypothesis $\eta(-I)=(-I)^ e$.

 Let $\tilde\eta_e\colon G_{k}\ra \GL(V)/\langle -I\rangle $ be the group homomorphism naturally induced by $\eta_e$. It factors through $K_0k_\varepsilon$ by Proposition \ref{proposition: K0}, and therefore $\eta_e$ factors through an at most quadratic extension of $K_0k_\varepsilon$.
\end{proof}

\section{Scenarios of applicability}\label{section: scenarios}

Let $A_0$ be an abelian variety defined over $k_0$ satisfying Hypothesis \ref{hypothesis: potGL2}. In this section, we use the description of the Sato--Tate group of $A_0$ achieved in \S\ref{section: STgroups} to prove the Sato--Tate conjecture in certain cases. The two main theorems of this section generalize \cite[Thm 3.4 and Thm 3.6]{Tay19}. The proofs build heavily on  those in \cite{Tay19}, which in turn are deeply inspired by those in \cite{Joh17}. Many ideas are in fact reminiscent of the seminal works \cite{HSBT10} and \cite{Har09}.

\begin{theorem}\label{theorem: powersecs}
Suppose that $k_0$ is a totally real number field and that $A_0$ is an abelian variety defined over $k_0$ of dimension $g\geq 1$ which is $\Qbar$-isogenous to the power of either:
\begin{enumerate}[i)]
\item an elliptic curve $B$ without CM; or 
\item an abelian surface $B$ with QM.
\end{enumerate}
Suppose that the endomorphism field $K_0$ of $A_0$ is a solvable extension of $k_0$. Then Conjecture~\ref{conjecture: ST} holds. 
\end{theorem}

\begin{proof}
The setting of the theorem is that of an abelian variety $A_0$ satisfying Hypothesis \ref{hypothesis: main} with $M=\Q$. In particular, we have $k=k_0$. By Theorem \ref{theorem: inddec}, there is an isomorphism of $\Qbar_\ell[G_{k_0}]$-modules
$$
V_\ell(A)\otimes \Qbar_\ell \simeq V_\ell(B)^{\alpha_B} \otimes_{\Qbar_\ell} V_\ell(B,A)^{\alpha_B}\,, 
$$
where $V_{\ell}(B,A)^{\alpha_B}$ has dimension $g$. 
We want to show that the partial $L$-function
\begin{align}\label{eq: L of symm e}
  L^S(\Symm^e\otimes \eta, A_0, s)
\end{align}
as defined in \eqref{equation: Lfunctype} is invertible as long as not both $e=0$ and $\eta$ is trivial. We may assume that $e\geq 1$, since otherwise we have a partial Artin $L$-function for which the result is well known. 

To show invertibility, we will apply the Taylor--Brauer reduction method of \cite{HSBT10} closely following the presentation of \cite{MM09}. We first invoke Theorem~\ref{theorem: automorphy} to obtain a Galois extension\footnote{It is not really necessary to assume that $k'/k$ is linearly disjoint from $K_0$ over $k_0$.} $k'/k$ such that 
$$
\Rr_e|_{G_{k'}}\,,\qquad \text{where } \Rr_e:=\Symm^e(V_\lambda(B)^{\alpha_B})_\lambda\,,
$$ 
is automorphic. Note that the partial $L$-function of \eqref{eq: L of symm e} is the normalized partial $L$-function $L^S(\Rr_e\otimes \eta_e,s)$ attached to the  weakly compatible system of $\lambda$-adic representations $\Rr_e\otimes \eta_e$.

Set $L=k'K_e$, where $K_e$ is the field introduced in Lemma \ref{lemma: eta_e}, and inflate $\eta_e$ to a representation of $\Gal(L/k_0)$. By Brauer's induction theorem, we may write $\eta_e$ as a finite sum
$$
\eta_e=\bigoplus_{i} c_i \Ind^{E_i}_{k_0}(\chi_i)\,,
$$ 
where $c_i$ is an integer, $E_i/k_0$ is a subextension of $L/k_0$ such that $\Gal(L/E_i)$ is solvable, and $\chi_i\colon \Gal(L/E_i)\rightarrow \C^\times$ is a character. We therefore have
$$
L^S(\Rr_e\otimes \eta_e, s)=\prod_i L^S(\Rr_e|_{G_{E_i}}\otimes \chi_i, s)^{c_i}\,,
$$
and it suffices to show that the $L$-functions $L^S(\Rr_e|_{G_{E_i}}\otimes \chi_i, s)$ are invertible.

By assumption, $K_0/k_0$ is solvable, and thus so is $K_e/k_0$. Therefore $L/k'$ is solvable. Then automorphic base change \cite{AC89} implies that $\Rr_e|_{G_{L}}$ is automorphic. Since $L/E_i$ is solvable, automorphic descent implies that 
$\Rr_e|_{G_{E_i}}$ is automorphic. Via Artin reciprocity, we may interpret $ \chi_i$ as a Hecke character of $E_i$, and deduce that 
\begin{equation}\label{equation: automsyst}
\Rr_e|_{G_{E_i}}\otimes \chi_i
\end{equation}
is automorphic. This implies that  $L^S(\Rr_e|_{G_{E_i}}\otimes \chi_i, s)$ is invertible.
\end{proof}

\begin{remark}\label{remark: CMB}
A statement analogous to Theorem \ref{theorem: powersecs} holds true when $B$ is a CM elliptic curve. In this case, the solvability assumption on $K_0/k_0$ is not necessary, since the automorphicity of Hecke characters is well known. We will however disregard the CM setting in this section, since it has already been treated in \cite[\S3]{Joh17}.
\end{remark}

\begin{theorem}\label{theorem: powerssurfs}
Suppose that $k_0$ is a totally real number field and that $A_0$ is an abelian variety defined over $k_0$ of dimension $g\geq 1$ which is $\Qbar$-isogenous to the power of either:
\begin{enumerate}[i)]
\item an abelian surface $B$ wih RM; or 
\item an abelian fourthfold $B$ with QM\footnote{Recall that in our terminology this means that $\End(B)$ is a quaternion algebra over a quadratic number field.}.
\end{enumerate}
Suppose that the endomorphism field $K_0$ of $A_0$ is a solvable extension of $k_0$ and that the field extension $k/k_0$ from Theorem \ref{theorem: inddec} is trivial. Then Conjecture~\ref{conjecture: ST} holds.
\end{theorem}

\begin{proof}
The hypotheses of the theorem are a reformulation of the assumption that~$A_0$ satisfies Hypothesis \ref{hypothesis: main} and that $M$ is a (real) quadratic field. 

Since $k=k_0$, we have that $M=M_0$. By Theorem \ref{theorem: inddec}, we have an isomorphism of $\Qbar_\ell[G_{k_0}]$-modules
$$
V_\ell(A)\otimes \Qbar_\ell \simeq V_{\lambda}(B)^{\alpha_B}\otimes V_{\lambda}(B,A)^{\alpha_B}\oplus
 V_{\overline\lambda}(B)^{\alpha_B}\otimes V_{\overline\lambda}(B,A)^{\alpha_B}\,,
$$
where $\lambda,\overline\lambda$ are attached to extensions to $F$ of the two distinct embeddings of $M_0$ into $\Qbar_\ell$. Note that $V_{\lambda}(B,A)^{\alpha_B}$ has dimension $g/2$ as a $\Qbar_\ell$-vector space. It will suffice to show that the partial $L$-function
$$
L^S(\Symm^{e_1} \otimes \Symm^{e_2} \otimes \eta, A_0, s)
$$
attached to \eqref{equation: reptyp2} is invertible whenever $e_1>0$ or $e_2>0$. Invoke Theorem \ref{theorem: automorphy} to obtain a Galois extension $k'/k_0$ such that $\Rr_{e_1}|_{G_{k'}}$ and $\overline\Rr_{e_2}|_{G_{k'}}$ are automorphic, where
$$
\Rr_{e_1}:=\Symm^{e_1}(V_{\lambda}(B)^{\alpha_B})_\lambda\,,\qquad
\overline\Rr_{e_2}:=\Symm^{e_2}(V_{\overline\lambda}(B)^{\alpha_B})_\lambda\,.
$$ 
 Set $L=k'K_{e_1+e_2}$ and inflate $\eta_{e_1+e_2}$ to a representation of $\Gal(L/k_0)$. Note that
\begin{align*}
  L^S(\Symm^{e_1}\otimes\Symm^{e_2}\otimes \eta, A_0, s)= L^S(\Rr_{e_1}\otimes \overline\Rr_{e_2}\otimes \eta_{e_1+e_2},s) .
\end{align*}

As in the proof of Theorem \ref{theorem: powersecs}, by Brauer's induction theorem applied to $\eta_{e_1+e_2}$, there exist integers $c_i$, subextensions $E_i/k_0$ of $L/k_0$ with $\Gal(L/E_i)$  solvable, and characters $\chi_i\colon \Gal(L/E_i)\rightarrow \C^\times$ such that
\begin{align}\label{eq: e1+e2}
    L^S(\Rr_{e_1}\otimes \overline\Rr_{e_2}\otimes \eta_{e_1+e_2},s)=\prod_i L^S(\Rr_{e_1}|_{G_{E_i}}\otimes \overline\Rr_{e_2}|_{G_{E_i}}\otimes \chi_i,s)^{c_i}\,.
    \end{align}
As in the proof of Theorem \ref{theorem: powersecs} we have that $L/k'$ is solvable. Using automorphic base change and automorphic descent, we find that 
\begin{equation}\label{equation: automsyst2}
  \Rr_{e_1}|_{G_{E_i}}\text{ and } \overline\Rr_{e_2}|_{G_{E_i}}\otimes \chi_i
\end{equation}
are automorphic. The nonvanishing of the $L$-function attached to \eqref{eq: e1+e2} follows from \cite[Theorem 5.2]{Sha81}. The holomorphicity follows from the fact that the systems of \eqref{equation: automsyst2} do not become isomorphic after a finite base change, as granted by Proposition \ref{proposition: noiso}. 
\end{proof}

\begin{remark}
When $g\leq 3$, the hypothesis that $K_0/k_0$ be solvable in Theorem \ref{theorem: powersecs} and Theorem \ref{theorem: powerssurfs} is always satisfied. This follows from the classification results achieved in \cite{FKRS12} and \cite{FKS20}. 
\end{remark}

\subsection*{Examples: modular abelian varieties}
A natural source of examples of abelian varieties satisfying Hypothesis \ref{hypothesis: main} are the modular abelian varieties associated to modular forms by the Eichler--Shimura construction. Let $f=\sum a_mq^m\in S_2(\Gamma_1(N))$ be a non-CM newform of Nebentypus $\varepsilon$, and let $F_f=\Q(\{a_m\}_m)$ be the number field generated by its Fourier coefficients. Put $g=[F_f :\Q]$. There exists an abelian variety $A_f$ defined over $\Q$ of dimension $g$ which is uniquely characterized up to isogeny by the equality of $L$-functions
\begin{align*}
  L(A_f,s)=\prod_{\sigma: F_f\hookrightarrow \C}L(f^\sigma,s),
\end{align*}
and satisfying that $\End(A_f)\simeq F_f$. The variety $A_f$ is simple, but it may not be geometrically simple. The structure of the base change $A_{f,\Qbar}$ was determined by Ribet \cite{Rib92} and Pyle \cite{Pyl02}. They proved that $A_{f,\Qbar}\sim B^d$ for some abelian variety $B/\Qbar$ satisfying that:
\begin{itemize}
\item $B$ is a $\Q$-variety, and
  \item $\End(B)$ is a central division algebra over a totally real field $M_f$ of Schur index $n\leq 2$ and $n[M_f:\Q]=\dim B$.
\end{itemize}
Moreover, the center $M_f$ of $\End(B)$ can be described in terms of $f$ as the field generated by all the numbers $a_m^2/\varepsilon(m)$ with $m$ coprime to $N$.

Denote by $K_f$ the smallest field of definition of $\End(A_{f,\Qbar})$ (this is the field called $K_0$ in \S \ref{section: td}). The extension $K_f/\Q$ is abelian (cf. \cite[Proposition 2.1]{GL01}), hence in particular solvable.

All these properties of the varieties $A_f$ give the following consequence of Theorems \ref{theorem: powersecs} and \ref{theorem: powerssurfs}.

\begin{corollary}\label{corollary modular forms}
  Let $f = \sum a_m q^m\in S_2(\Gamma_1(N))$ be a newform of nebentype~$\varepsilon$. If $f$ is non-CM, suppose that the number field $M_f=\Q(\{a_m^2/\varepsilon(m)\}_{(m,N)=1})$ has degree at most~$2$ over~$\Q$. Then the Sato--Tate conjecture is true for $A_f$.
\end{corollary}
 Examples of these modular forms are certainly abundant even for small levels $N$. For example, in the tables of \cite{Que09} (the complete tables are available at \cite{Que12}), where levels up to $500$ are considered, one finds many examples of modular abelian varieties $A_f$ which are geometrically isogenous to powers of elliptic curves (\cite[\S 4.1]{Que12}), abelian surfaces with RM by a quadratic field $M_f$  (\cite[\S 4.2]{Que12}), abelian surfaces with QM  by a quaternion algebra over $\Q$ (\cite[\S 5.1]{Que12}) or abelian fourfolds with QM by a quaternion algebra over a quadratic field $M_f$ (\cite[\S 5.2]{Que12}).

 Suppose that $f\in S_2(\Gamma_0(N),\varepsilon)$ is a modular form satisfying the hypotheses of Corollary \ref{corollary modular forms}, and let $\chi$ be a Dirichlet character. Then the twisted modular form $h=f\otimes\chi$ has nebentype character $\varepsilon_h=\varepsilon\chi^2$, and its Fourier coefficients $b_m$ satisfy that $b_m=a_m\chi(m)$ for $m$ coprime to the conductor of $\chi$. Therefore, $b_m^2/\varepsilon_h(m)=a_m^2/\varepsilon(m)$ so the field $M_h$ coincides with $M_f$ and therefore $h$ also satisfies the hypothesis of Corollary \ref{corollary modular forms}. Since $\dim A_h=[F_h\colon \Q]$ and $F_h$ contains the field generated by the values of $\varepsilon_h$, we see that Corollary \ref{corollary modular forms} proves the Sato--Tate conjecture for varieties of the form $A_h$ of arbitrarily large dimension.

\end{document}